\documentclass[11pt]{amsart}

\usepackage{amsmath, amsthm, amssymb, amsfonts, enumerate}
\usepackage[colorlinks=true,linkcolor=blue,urlcolor=blue]{hyperref}
\usepackage{dsfont}
\usepackage{color}
\usepackage{geometry}
\usepackage{todonotes}

\newtheorem{theorem}{Theorem}[section]
\newtheorem{remark}[theorem]{Remark}

\newtheorem{proposition}[theorem]{Proposition}
\newtheorem{corollary}[theorem]{Corollary}

\newcommand{\LL}{ \mathcal{L}}

\title[Properties of Ratios of Hermite and Parabolic Cylinder Functions]{Universal Bounds and Monotonicity Properties of Ratios of Hermite and Parabolic Cylinder Functions}
\author[T. Koch]{Torben Koch}

\address{T.~Koch: Center for Mathematical Economics (IMW), Bielefeld University, Universit\"atsstrasse 25, 33615, Bielefeld, Germany}
\email{\href{mailto:t.koch@uni-bielefeld.de}{t.koch@uni-bielefeld.de}}

\date{\today}

\numberwithin{equation}{section}

\begin{document}
	
\begin{abstract} 
We obtain so far unproved properties of a ratio involving a class of Hermite and parabolic cylinder functions. Those ratios are shown to be strictly decreasing and bounded by universal constants. Differently to usual analytic approaches, we employ simple purely probabilistic arguments to derive our results. In particular, we exploit the relation between Hermite and parabolic cylinder functions and the eigenfunctions of the infinitesimal generator of the {Ornstein-Uhlenbeck process}. As a byproduct, we obtain {Tur\'an type inequalities}.  
\end{abstract}

\maketitle

\smallskip


{\textbf{Key words}}: Hermite functions; parabolic cylinder functions; Tur\'an type inequalities; Ornstein-Uhlenbeck process.







\section{Introduction}
\label{introduction}
Consider the ordinary differential equation (ODE) \begin{align}\label{eq1}
u''(x)-2xu'(x)+2\nu u(x)=0,\quad\nu<0,\,x\in\mathbb{R}.
\end{align}
Following Section 10.2 in \cite{Lebedev}, the solutions to \eqref{eq1} are called Hermite functions and denoted by $H_\nu$. They are closely connected to parabolic cylinder functions. In fact, letting $\Gamma$ be the Euler's Gamma function, the parabolic cylinder function $D_\nu$, introduced in \cite{Whittaker}, admits the representation (cf. Section 8.3 in \cite{Erdelyi})
\begin{align}\label{eq3}
D_\nu(x):=\frac{e^{-\frac{x^2}{4}}}{\Gamma(-\nu)}\int_{0}^{\infty}t^{-\nu-1}e^{-\frac{t^2}{2}-xt}dt,\quad x\in\mathbb{R},
\end{align}
and satisfies  
 \begin{align}\label{identityCyHer}
D_\nu(x)=2^{-\frac{\nu}{2}}e^{-\frac{x^2}{4}}H_\nu\left(\frac{x}{\sqrt{2}}\right),\quad x\in\mathbb{R}.
\end{align}
In this paper, we study properties of the ratio $\mathcal{R}_\nu:\mathbb{R}\mapsto\mathbb{R}$, where \begin{align}\label{Ratio}
\mathcal{R}_\nu(x):=\frac{(H_{\nu-1}(x))^2}{H_{\nu}(x)H_{\nu-2}(x)},\quad x\in\mathbb{R},
\end{align}
and thanks to \eqref{identityCyHer}, our results carry over to the ratio of $D_\nu$ as well. In particular, we show that $\mathcal{R}_\nu$ is strictly decreasing, and we derive its best possible upper and lower bounds. 

The ratio \eqref{Ratio} is closely related to the so-called Tur\'an types inequalities. Those inequalities have been discovered in 1941 by P. Tur\'an (published in 1950, see \cite{Turan}) for Legendre Polynomials $P_n$, $n\in\mathbb{N}$, and for those functions they read as
\begin{align}\label{TuranIneq}
 P_{n-1}(x)P_{n+1}(x)-P_n^2(x)<0,\quad\text{for all } x\in(-1,1).
\end{align}
Notice that the validity of \eqref{TuranIneq} was first proved by G. Szeg\"o in 1948 (see \cite{Szego}). Since then, inequalities of this form have attracted a lot of attention, and have been proved to be valid for other polynomials such as Hermite (obtained from Hermite functions by taking $\nu\in\mathbb{N}$), Jacobi, Laguerre or ultraspherical  polynomials (see \cite{Gasper,Szego}, among others), and for special functions as (modified) Bessel, Gamma, parabolic cylinder or 
hypergeometric functions (see \cite{Alzer,baricz3, baricz4, baricz2, baricz,Thiruvenkatachar}, among many others). Applications of Tur\'an type inequalities can be found in many fields, ranging from biophysics (see \cite{baricz5} and the references therein) to information theory (see \cite{McEliece}) and stochastic control (see \cite{Becherer,Koch}).

Properties of ratios of special functions as in \eqref{Ratio} have also gained interest in recent years. In \cite{Sitnik}, conjectures about the monotonicity of a ratio associated to exponential series sections are formulated. Those conjectures are then proved in \cite{SitnikMehrez1,SitnikMehrez2} for classical Kummer and Gauss hypergeometric functions, as well as for the so-called $q$-Kummer confluent hypergeometric and $q$-hypergeometric functions. Moreover, the monotonicity of a ratio like \eqref{Ratio} associated to modified Bessel functions of the first and second kind has been proved to be valid and used for the proofs of Theorem 2.1 and Theorem 3.1 in \cite{baricz3}. Our focus on \eqref{Ratio} is motivated by an optimal liquidation problem in a financial market (see Remark 6.8 in \cite{Becherer}). Lower and upper bounds for $\mathcal{R}_\nu$ have already been studied by \cite{Segura}, but we are able to show that our bounds are the best possible ones, and this leads to a discrepancy between the results in \cite{Segura} and ours (see Remark \ref{Remark}).

In all the aforementioned references on Tur\'an type inequalities (see \cite{Alzer,baricz3, baricz4, baricz2, baricz,Gasper,SitnikMehrez1,SitnikMehrez2, Szego,Thiruvenkatachar,Turan}), the authors use purely analytic approaches to prove their results. Here, instead, we follow a completely different approach that uses probabilistic arguments, and leads to a simple and short proof of our results. In particular, we exploit the relation of Hermite functions to the eigenfunctions of the infinitesimal generator of the {Ornstein-Uhlenbeck process}.

The paper is organised as follows: in Section \ref{Sec:ProbCon}, we introduce and recall the properties of the Ornstein-Uhlenbeck process and point out its connections to Hermite and parabolic cylinder functions. Then, in Section \ref{sec:MonotonRatio}, the results from Section \ref{Sec:ProbCon} are used to prove the claimed properties of $\mathcal{R}_\nu$.

\section{Ornstein-Uhlenbeck Process and Hermite Functions}\label{Sec:ProbCon}
Let $(\Omega, \mathcal{F}, \mathbb{F}:=(\mathcal{F}_t)_{t\geq 0}, \mathbb{P})$ be a filtered probability space with a filtration $\mathbb{F}$ satisfying the usual conditions, and carrying a standard one-dimensional $\mathbb{F}$-Brownian motion $W$.

We now introduce the so-called Ornstein-Uhlenbeck process and we will recall some of its well known properties. The Ornstein-Uhlenbeck process has been introduced for the first time in \cite{Ornstein}, and in modern stochastic analysis it is defined as the unique strong solution to the stochastic differential equation
\begin{align}\label{Xuncontrolled}
dX^x_t = \left(\mu-X^x_t\right)dt + \sigma dW_t,\quad X^x_0 = x\in\mathbb{R},
\end{align}
for $\mu\in\mathbb{R}$ and $\sigma>0$. For any given initial value $x\in\mathbb{R}$, the process $X^x:=(X^x_t)_{t\geq 0}$ is Gaussian. In particular, equation \eqref{Xuncontrolled} admits the explicit solution
\begin{align}\label{SolOU}
X^x_t=xe^{-t}+\mu(1-e^{-t})+\int_{0}^{t}\sigma e^{s-t}dW_s,
\end{align} and it follows from \eqref{SolOU} that for any $t\geq 0$ $$X_t^x\sim\mathcal{N}\left(xe^{-t}+\mu\left(1-e^{-t}\right),\frac{\sigma^2}{2}\left(1-e^{-2t}\right)\right),$$ where $\mathcal{N}(\alpha,\gamma)$ denotes the Gaussian distribution function with mean $\alpha\in\mathbb{R}$ and variance $\gamma>0$.

The infinitesimal generator associated to $X^x$ is denoted by $\mathcal{L}$ and, for any $u:\mathbb{R}\mapsto\mathbb{R}$ s.t. $u\in C^2(\mathbb{R})$, it is defined by
\begin{align}
\left[\mathcal{L}u\right](x):=\lim\limits_{t\downarrow 0}\frac{\mathbb{E}\left(u(X_t^x)\right)-u(x)}{t},\quad x\in\mathbb{R}.
\end{align}
In particular, by an application of Dynkin's formula (cf.\ Theorem 7.4.1 in \cite{Oksendal}) and of the mean-value theorem, one obtains
$$\mathcal{L}u(x)=\frac{\sigma^2}{2}u''(x)+(\mu-x)u'(x),\quad x\in\mathbb{R}.$$
Given $\nu<0$, it is well known that the ODE $\mathcal{L}u+\nu u=0$ admits a strictly increasing positive fundamental solution $\psi$. The function $\psi$ can be expressed in terms of the cylinder function (see, e.g., p. 280 in \cite{Jeanblanc}) (up to a positive constant); that is 
\begin{align}\label{eq6}
\psi(x)&\propto e^{\frac{(x-\mu)^2}{2\sigma^2}}D_{\nu}\bigg(\frac{\mu-x}{\sigma}\sqrt{2}\bigg),\quad x\in\mathbb{R}.
\end{align}
In light of \eqref{identityCyHer} and \eqref{eq6}, from now on we thus identify the positive strictly increasing eigenvector of $\mathcal{L}$ with
\begin{align}\label{eq4}
\psi(x)&=H_\nu\left(\frac{\mu-x}{\sigma}\right),\quad x\in\mathbb{R}.
\end{align}


\section{The Main Result: Monotonicity of Ratios of Hermite Functions}\label{sec:MonotonRatio}
In this section, we use the link between the Ornstein-Uhlenbeck process and the Hermite functions in order to study the monotonicity of the ratio $\mathcal{R}_\nu$ from \eqref{Ratio}.


\begin{theorem}\label{MonProp}
	For all $\nu<0$, the function $\mathcal{R}_\nu$ as in \eqref{Ratio} is strictly decreasing.
\end{theorem}

\begin{proof}
The proof is organised in two steps. First, recalling $\psi$ from \eqref{eq4}, in \emph{Step 1} we prove that the function $\Psi:\mathbb{R}\mapsto\mathbb{R}$ such that
$$\Psi(x):=\frac{\psi'(x)^2}{\psi(x)\psi''(x)},\quad x\in\mathbb{R},$$ is strictly increasing. Then, in \emph{Step 2} we make the conclusion for $\mathcal{R}_\nu$.

\vspace{0.25cm}
\emph{Step 1.} Let $x,y\in\mathbb{R}$ be such that $y>x$, recall $X^x$ solving \eqref{Xuncontrolled}, and define the first hitting time of $X^x$ at level $y$ by $$\tau_y:=\inf\{t\geq 0:\, X_t^x\geq y\},\quad\mathbb{P}\text{-a.s.}$$ Direct calculations on \eqref{eq4} and the identity (cf., e.g., equation (10.4.4) in \cite{Lebedev})
\begin{align}\label{DerH}
H_\nu'(x)=2\nu H_{\nu-1}(x),\quad x\in\mathbb{R},\,\nu<0,
\end{align}
show that the $k$-th derivative of $\psi$, denoted by $\psi^{(k)}$, is a strictly increasing positive solution to $\left(\LL +(\nu-k)\right)\psi^{(k)}=0$. Now, since $\psi^{(k)}\in C^2(\mathbb{R})$ for any $k\in\mathbb{N}_0$, It\^o's formula (together with a standard localization argument) yields, after taking expectations,
\begin{align*}
\mathbb{E}\left[e^{(\nu-k)\tau_y}\psi^{(k)}\big(X_{\tau_y}^x\big)\right]=\psi^{(k)}(x),
\end{align*}
and hence 
\begin{align}\label{eq5}
\mathbb{E}\left[e^{(\nu-k)\tau_y}\right]=\frac{\psi^{(k)}(x)}{\psi^{(k)}(y)},
\end{align}
since $\tau_y<\infty$ $\mathbb{P}$-a.s.\ as $X^x$ is positively recurrent (cf.\ Appendix 1.24 in \cite{Borodin}). Now, H\"older's inequality yields
\begin{align}\label{eq2}
\mathbb{E}\left[e^{\nu\tau_y}\right]^{\frac{1}{2}}\mathbb{E}\left[e^{(\nu-2)\tau_y}\right]^{\frac{1}{2}}>\mathbb{E}\left[e^{(\nu-1)\tau_y}\right],
\end{align} which is strict since the function $f(z):=e^{\nu z}$ is not a multiple of the function $g(z):=e^{(\nu-2) z}$, and the random variable $\tau_y$ has a distribution which is absolutely continuous with respect to the Lebesgue measure on $\mathbb{R}_+$.
From both \eqref{eq5} with $k=0,1,2$ and \eqref{eq2}, we find $$\Psi(y)>\Psi(x).$$ Since $y>x$ were arbitrary, we have that $x\mapsto\Psi(x)$ is strictly increasing.

\vspace{0.25cm}
\emph{Step 2.} Exploiting the identities \eqref{eq4} and \eqref{DerH}, we find that for any $x\in\mathbb{R}$ $$\Psi(x)=\frac{\nu}{\nu-1}\mathcal{R}_\nu\left(\frac{\mu-x}{\sigma}\right).$$ Therefore, because $\nu<0$, we conclude by \emph{Step 1} that the function $\mathcal{R}_\nu$ is strictly decreasing.
\end{proof}

The following corollary gives the best possible bounds for $\mathcal{R}_\nu$. These in turn imply the Tur\'{a}n type inequality.
\begin{corollary}\label{CorBounds}
	For all $\nu<0$, the function $\mathcal{R}_\nu$ as in \eqref{Ratio} is such that 
	\begin{align}\label{Bounds}
	1<\mathcal{R}_\nu(x)<\frac{\nu-1}{\nu},\quad\text{for all }x\in\mathbb{R}.
	\end{align}
	In particular, the following Tur\'{a}n-type inequality holds:
	$$H_{\nu-1}(x)^2-H_{\nu}(x)H_{\nu-2}(x)>0,\quad x\in\mathbb{R}.$$
\end{corollary}

\begin{proof}
	Equations (10.6.4) and (10.6.7) in \cite{Lebedev} provide the asymptotic behavior of $H_\nu(x)$ for both (large) positive and (large) negative values of $x$. In particular, it holds that $$\lim\limits_{x\downarrow-\infty}\mathcal{R}_\nu(x)=\frac{\nu-1}{\nu},\quad\quad\lim\limits_{x\uparrow+\infty}\mathcal{R}_\nu(x)=1.$$ Thus, \eqref{Bounds} follows from the strict monotonicity of $x\mapsto\mathcal{R}_\nu(x)$ proved in Theorem \eqref{MonProp}.
\end{proof}
Since, by \eqref{identityCyHer}, we have $\mathcal{R}_\nu\left(\frac{x}{\sqrt{2}}\right)=\frac{D_{\nu-1}(x)^2}{D_\nu(x)D_{\nu-2}(x)}$ for any $x\in\mathbb{R}$, the next proposition easily follows from Theorem \ref{MonProp} and Corollary \ref{CorBounds}.
\begin{proposition}\label{CylinderBound}
	For all $\nu<0$, the function $\widetilde{\mathcal{R}}_\nu:\mathbb{R}\mapsto\mathbb{R}$ defined as 
	\begin{align*}
	\widetilde{\mathcal{R}}_\nu\left(x\right):=\frac{D_{\nu-1}(x)^2}{D_\nu(x)D_{\nu-2}(x)},\quad x\in\mathbb{R},
	\end{align*}
	is strictly decreasing. Moreover, it holds $$1<\widetilde{\mathcal{R}}_\nu(x)<\frac{\nu-1}{\nu},\quad\text{for all }x\in\mathbb{R}.$$
\end{proposition}

\begin{remark}\label{Remark}
 It is worth mentioning that lower and upper bounds of the function $\widetilde{\mathcal{R}}_\nu$ associated to the parabolic cylinder function $U_{-\nu-\frac{1}{2}}(x)=D_{\nu}(x)$ (see (19.3.1) in \cite{Abra}) have also been derived in \cite{Segura}. In that paper, the right-hand side of equation (28) (see also Remark 1 in \cite{baricz}) yields an upper bound for $\widetilde{\mathcal{R}}_\nu$ which is strictly less than the one we have obtained in Proposition \ref{CylinderBound}. Given that our upper bound is optimal by the proved strict monotonicity of $\widetilde{\mathcal{R}}_\nu$, it seems that there is something fishy in eq. (28) of \cite{Segura}. Also, a simple numerical analysis seems to contradict the upper bound found in \cite{Segura}. 
\end{remark}
\vspace{1cm}

\indent \textbf{Acknowledgments.} Financial support by the German Research Foundation (DFG) through the Collaborative Research Centre 1283 ``Taming uncertainty and profiting from randomness and low regularity in analysis, stochastics and their applications'' is gratefully acknowledged by the author. I wish to thank Giorgio Ferrari for useful discussions.


\end{document}